\documentclass[12pt,english]{amsart}
\usepackage[T1]{fontenc}
\usepackage[latin9]{inputenc}
\usepackage{babel}
\usepackage{url}
\usepackage{color}

\theoremstyle{plain}
\newtheorem{thm}{Theorem}
\newtheorem{pro}[thm]{Proposition}
\newtheorem{lem}[thm]{Lemma}
\newtheorem{cor}[thm]{Corollary}

\theoremstyle{definition}
\newtheorem{example}[thm]{Example}
\newtheorem*{question}{Question}

\DeclareMathOperator{\GF}{GF}
\DeclareMathOperator{\PG}{PG}

\begin{document}

\title{Some necessary conditions for vector space partitions}

\author{Juliane Lehmann}
\address{J. Lehmann\\
Universit\"at Bremen\\
Department of Mathematics\\
Bibliothekstrasse 1 - [MZH]\\
28359 Bremen\\
Germany }
\email{jlehmann@math.uni-bremen.de}
\author{Olof Heden}
\address{O. Heden\\
Department of Mathematics\\
KTH\\
S-100 44 Stockholm\\
Sweden}
\email{olohed@math.kth.se}

\thanks{First author supported by grant KAW 2005.0098 from the Knut and Alice Wallenberg Foundation.}

\begin{abstract}
Some new necessary conditions for the existence of vector space partitions are derived. They are applied to the problem of finding the maximum number of spaces of dimension $t$ in a vector space partition of $V(2t,q)$ that contains $m_d$ spaces of dimension $d$, where $t/2<d<t$, and also spaces of other dimensions.  It is also discussed how this problem is related to maximal partial $t$-spreads in $V(2t,q)$. We also give a lower bound for the number of spaces in a vector space partition and verify that this bound is tight. 
\end{abstract}

\maketitle

\section{Introduction}\label{sec:1}

Let $V$ denote a vector space $V(n,q)$ of dimension $n$ over the finite field $\GF(q)$. A \emph{vector space partition} $\mathcal{P}$ of $V$ 
is a collection of subspaces $U_{i}$, $i\in I$, of $V$, all of positive dimension, with the property that every nonzero vector of $V$ is contained in a unique
member of $\mathcal P$.

If for $d\in\{1,2,\ldots,k\}$ the vector space partition $\mathcal{P}$ contains $m_{d}$ subspaces of dimension
$d$ and no subspaces of dimension higher than $k$, then the $k$-tuple $(m_{k},m_{k-1},\dots,m_{1})$ is called
the \textit{type} of the partition, and we say that $\mathcal{P}$
is an $(m_{k},m_{k-1},\dots$, $m_{1})$-\textit{partition}
of $V$.
By counting the nonzero vectors in $V(n,q)$, it is clear that
\begin{equation*}
\sum_{d=1}^k m_d(q^d-1)=q^n-1\,.
\end{equation*}
This necessary condition for the existence of a vector space partition of a certain type will here be called the \emph{first packing condition}.

As any two distinct members $U$ and $U'$ of a vector space partition $\mathcal P$ of $V(n,q)$ span a subspace the dimension of which is the sum of the dimensions of $U$ and $U'$, it is clear that
\begin{equation*}
\dim(U)+\dim(U')\leq n\,.
\end{equation*} 
This well known necessary condition for the existence of a vector space partition will here be called the \emph{dimension condition}.

The aim of this paper is twofold: first, we derive further necessary conditions for the existence of vector space partitions of certain types; then we apply them to the study of maximal partial $t$-spreads and vector space partitions in general.

A \emph{partial $t$-spread} $\mathcal{S}$ in the finite vector space $V=V(n,q)$ is here a collection of $t$-dimensional subspaces of $V$ with the property that the intersection of any two members of $\mathcal{S}$ is the zero vector.\footnote{The terminology differs among authors: in some contexts, a partial $t$-spread is regarded as a collection of mutually disjoint $t$-dimensional subspaces of the finite projective space $\PG(n,q)$. In our notation, this would be a partial $(t+1)$-spread in $V(n+1,q)$. Traditionally, a projective $1$-spread in $\PG(3,q)$ is called a \emph{spread}.} If further $\mathcal{S}$ constitutes a vector space partition of $V$, then $\mathcal{S}$ is called a \emph{$t$-spread} in $V$. A \emph{maximal partial $t$-spread} $\mathcal{S}$ in $V$ is a partial $t$-spread with the property that every $t$-dimensional subspace of $V$ has a nontrivial intersection with at least one member of $\mathcal{S}$.

It follows from the first packing condition that if there exists a $t$-spread in $V(n,q)$, then $t$ divides $n$. It is well known (and it will be demonstrated in Section \ref{sec:5}) how to construct $t$-spreads in $V(kt,q)$. However, the complete spectrum of the sizes of maximal partial $t$-spreads is far from known. 

We will mainly apply the new necessary conditions to the case $k=2$, that is, to $t$-spreads and to those partial $t$-spreads in $V(2t,q)$ that are parts of vector space partitions also containing spaces of other, relatively high, dimensions. 
One of our main results, Theorem \ref{thm:lowbound-a}, is the identification of a function $\mathrm{R}_q(t,d,m)$ with the property that for every vector space partition $\mathcal P$ of type $(m_t,m_{t-1},\dots,m_1)$ in $V(2t,q)$, and with $m_t\neq0$,
\begin{equation*}
m_d<\frac{q^t-1}{q^{t-d}-1}\qquad\Longrightarrow\qquad m_t+m_d\leq q^t+1+\mathrm{R}_q(t,d,m_d)\,.
\end{equation*}
For instance, this enables us to exclude the possibility of a vector space partition in $V(8,2)$ of the type $(13,6,0,18)$, as $R_2(4,3,6)=15/9$.\footnote{The exclusion of a vector space partition of $V(8,2)$ of the type $(13,6,6,0)$ was the contribution of the first author of this paper to the paper \cite{hedenpapa}. The search for a generalization of the method used to prove the nonexistence of that particular example was the starting point of the present study.}
It must be remarked that the size of a $t$-spread in $V(2t,q)$ is $q^t+1$, and, as will be discussed in Example~\ref{ex:12} of Section~\ref{sec:4.2}, if $d$ is ''close'' to $t$ and $m$ is not too big, then $\mathrm{R}_q(t,d,m)<1$.

One very natural way to construct a vector space partition of the type $(m_t,m_{t-1},\dots,m_1)$ in $V(2t,q)$ is to start with a $t$-spread, and then partition $a=q^t+1-m_t$ of the $t$-dimensional spaces of the $t$-spread into subspaces. If $m_d>a$ for some $d>t/2$, then, by the dimension condition, this procedure is not possible to perform. However, there are examples of vector space partitions $\mathcal P$ such that the $t$-dimensional subspaces of $\mathcal P$ can be extended to a $t$-spread despite the fact that $m_d$ might be bigger than $a$; both constructions presented in Section~\ref{sec:5} provide such examples. Thus, the situation is rather complex. Nevertheless, in Corollary~\ref{corol:3} in Section~\ref{sec:4.1} we will prove that if both $d>t/2$ and $a<m_d\leq q^{\lceil d/2\rceil}+1$ hold, then the spaces of dimension $t$ in $\mathcal S$ cannot be embedded in any $t$-spread.  This shows the connection to maximal partial $t$-spreads in $V(2t,q)$, as then there exists a maximal partial $t$-spread of a size bigger than or equal to $m_t$ but smaller than $q^t+1$.

The paper is organized in the following way:

Section~\ref{sec:2} reviews one further necessary condition for vector space partitions, beside those mentioned above. We derive our new necessary conditions in Section~\ref{sec:3} and apply them to find the relation to partial $t$-spreads discussed above in Section~\ref{sec:4}. In Section~\ref{sec:5} we give two new constructions of vector space partitions, and we use them to discuss how far from being tight our results of Section~\ref{sec:4} are. In Section~\ref{sec:6} we discuss the consequences of our necessary conditions in connection with two distinct derivation processes: if a partition $\mathcal{P}$ of a given type exists, then
some other partitions $\mathcal{P}_{i}$ of a certain type of other
vector spaces $V_{i}$ must exist as well. These derivation procedures will give natural extensions of our new necessary condition. Finally, these new extensions are evaluated in a computer search.

\section{The tail condition}\label{sec:2}

There are no known general necessary and sufficent conditions for the existence of a vector space partition of the type $(m_k,m_{k-1},\dots,m_1)$ in a finite vector space $V(n,q)$. In this section we present one further necessary condition known as the tail condition, in order to give an overview of all known necessary conditions that later will be compared 
in strength to our new necessary conditions.

The \textit{tail} of a vector space partition $\mathcal{P}$ is 
defined to be the set of members of $\mathcal{P}$ that have the lowest dimension
among the members of $\mathcal{P}$. In \cite{h} the following bounds
for the size $m$ of the tail of a vector space partition were derived:
\begin{pro}\label{pro:tail}
Let $d_1$ be the lowest and $d_2$ be the second lowest among the dimensions in a vector space partition $\mathcal P$ of $V(n,q)$, and let $m$ denote the number of spaces of dimension $d_1$ in $\mathcal P$. Then,
\begin{enumerate}
\item if $d_{2}<2d_{1}$, then $m\geq q^{d_{1}}+1$;
\item if $d_{2}\geq2d_{1}$, then either $d_{1}$ divides $d_{2}$ and $m=(q^{d_{2}}-1)/(q^{d_{1}}-1)$
or $m>2q^{d_{2}-d_{1}}$.
\end{enumerate}
In case that $q^{d_{2}-d_{1}}$ divides $m$, these bounds can be
improved in all cases except $q=2,d_{1}=1,d_{2}=2$ (in which there is no improvement) to
\begin{enumerate}
\setcounter{enumi}{2}
\item If $d_{2}<2d_{1}$, then $m\geq q^{d_{2}}-q^{d_{1}}+q^{d_{2}-d_{1}}$.
\item If $d_{2}\geq2d_{1}$, then $m\geq q^{d_{2}}$. 
\end{enumerate}
\end{pro} 
These conditions will be called the
\textit{tail conditions}. It was shown in \cite{h} that the bounds (i), (ii) and (iv) are tight. 
It follows from an example in \cite{hedenpapa} that the bound in (iii) cannot be improved in general.

Let us also mention the following 
result that was proved in \cite{heden09}. 
\begin{pro}
The 
packing condition and the dimension condition together with the condition 
\[
\dim(U_i)\geq c,\quad\text{ for }i\leq q,\
\]
are necessary and sufficient for the existence of a vector space partition $U_1$, $U_2$, ..., $U_k$ of $V(n,q)$, if $\dim(U_{q+1})=\dim(U_{q+2})=\ldots=\dim(U_k)=c$.

\end{pro}

\section{Some further necessary conditions}\label{sec:3}

In this section, let $V$ denote the vector space $V(n,q)$.
An $(n-1)$-dimensional subspace $H$ of  
$V$ is called a \textit{hyperplane} in $V$. Most of the necessary
conditions that we derive in this section will be consequences of the
following 
lemma: 
\begin{lem} 
Any subspace of dimension $d$ of
the vector space $V$ is contained in exactly $(q^{n-d}-1)/(q-1)$
hyperplanes in $V$. 
\label{lem:1}\end{lem}

Motivated by this, let \[h(d;n,q):=\max\left\{ 0,\frac{q^{n-d}-1}{q-1}\right\}\,.\]

Let $U$ be any subspace of $V$. Any hyperplane $H$ either contains
$U$ or intersects $U$, in a subspace of dimension $\dim(U)-1$. 
For an $(m_{k},m_{k-1},\ldots,m_{1})$-partition $\mathcal{P}$ of
$V$, if $H$ contains $b_{d}$ of the subspaces of dimension $d$
we call $H$ a \emph{hyperplane of type} $b$, where $b$ is the tuple $(b_{k},b_{k-1},\ldots,b_{1})$,
and denote by $s_{b}$ the number of hyperplanes of type $b$. Then,
if $s_{b}>0$, we derive that \[
\sum_{d=1}^{k} b_{d}(q^{d}-1)+\sum_{d=1}^{k}(m_{d}-b_{d})(q^{d-1}-1)=q^{n-1}-1.\]
In particular we have the following, by using the above relation and
the first packing condition: 
\begin{equation}
s_{b}\neq0\quad\Longrightarrow\quad\sum_{d=1}^{k} b_{d}q^{d}=\sum_{d=1}^{k}m_{d} - 1.\label{eq:sec-pack-cond}\end{equation}
This equation is here called \textit{the second packing condition}. It is a linear Diophantine equation; let $B$ be the set
of all its 
solutions $b$ with $0\leq b_d \leq m_d$ for all $d\in\{1,2,\ldots,k\}$.

Since there are $(q^{n}-1)/(q-1)$ hyperplanes in $V$, we have 
\begin{equation}
\frac{q^{n}-1}{q-1} =  \sum_{b\in B}s_{b}.\label{eq:sum-sb-0}\end{equation}

If $b_{d}>1$ for some $d\geq\frac{n}{2}$, then we obtain $s_{b}=0$ from the dimension condition. The following lemma
generalizes Equation~\eqref{eq:sum-sb-0}, yielding a family of equalities and inequalities.

\begin{lem}Let $l\in\mathbb{N}$, let $d_i \in \{1,2,\ldots,n-1\}$ with $d_{1}<d_{2}<\ldots<d_{l}$, and
 $k_i\in\{1,2,\ldots,m_{d_i}\}$ for $i \in \{1,2,\ldots l\}$. Then \begin{multline}\label{eq:sum-sb-gen}
\left(\prod_{i=1}^{l}\binom{m_{d_{i}}}{k_{i}}\right)h(\sum_{i=1}^{l}k_{i}d_{i};n,q)\leq \\
\leq\sum_{b\in B}\left(\prod_{i=1}^{l}\binom{b_{d_{i}}}{k_{i}}\right)s_{b}
\leq\left(\prod_{i=1}^{l}\binom{m_{d_{i}}}{k_{i}}\right)h(d_{\textrm{min}};n,q),\end{multline}
where $d_{\textrm{min}}=\begin{cases}
2d_{l} & \textrm{if }k_{l}\geq2\\
d_{l}+d_{l-1} & \textrm{if }k_{l}=1\textrm{ and }l\geq2\\
d_{1} & \textrm{if }l=k_{1}=1\\
0 & \textrm{if }l=0\end{cases}$.\end{lem}

\begin{proof}We will double count all tuples $(H,\mathcal{U}_{d_{1}},\mathcal{U}_{d_{2}},\ldots,\mathcal{U}_{d_{l}})$,
where $H$ is a hyperplane, $\mathcal{U}_{d_{i}}$ is a subset of
the $d_{i}$-dimensional subspaces in $\mathcal{P}$ with cardinality
$k_{i}$ and $U\subseteq H$ for all $U\in\mathcal{U}_{d_{i}}$. We
obtain the expression in the middle, since each hyperplane $H$ contributing
to $s_{b}$ gets counted once for each choice of the sets $\mathcal{U}_{d_{i}}$
entirely from those subspaces in the partition contained in $H$. 

On the other hand, all different choices of $(\mathcal{U}_{d_{i}})_{i=1}^l$
yield subspaces $U$, where $U=\sum_{i=1}^{l}\sum_{U_{\alpha}\in\mathcal{U}_{d_{i}}}U_{\alpha}$.
Each of these subspaces is contained in $h(\dim(U);n,q)$
different hyperplanes, and $\dim(U)$ is at most $\sum_{i}k_{i}d_{i}$,
if all the sums are direct, and at least $d_{\textrm{min}}$, because
the sum of any \emph{two} $U_{\alpha}$ is direct, in particular of
the two of highest dimension.\end{proof}

\begin{cor}For any $1\leq d<d'\leq n-2$ such that $m_d,m_{d'}>0$,\begin{eqnarray}
\sum_{b\in B}b_{d}s_{b} & = & m_{d}h(d;n,q)\label{eq:sum-sb-i}\\
\sum_{b\in B}\binom{b_{d}}{2}s_{b} & = & \binom{m_{d}}{2}h(2d;n,q)\label{eq:sum-sb-2i}\\
\sum_{b\in B}b_{d}b_{d'}s_{b} & = & m_{d}m_{d'}h(d+d';n,q).\label{eq:sum-sb-i-j}\end{eqnarray}
\label{cor:1}\end{cor}

Note that Equations~\eqref{eq:sum-sb-0}
and \eqref{eq:sum-sb-gen} together with the condition that all $s_{b}$
are nonnegative define a polytope $P_m$ in a real space with basis $B$.
Then all the necessary conditions are fulfilled iff the integer hull $I(P_m)=\mathbb{Z}^{|B|}\cap P_m$ of 
this polytope is nonempty, which can be checked for example using 
LattE macchiato \cite{latte}.

\section{An application}\label{sec:4}

In this section we will apply the necessary conditions given in Section \ref{sec:3} to the following problem, already touched upon in the introduction: 

Assume that you are concerned with vector space partitions $\mathcal{P}$ of $V(2t,q)$ into $q^t+1-a$ spaces of dimension $t$ and $m_d$ spaces of dimension $d$ larger than $t/2$, and the remaining spaces of dimension 1. As any vector space of dimension $t$ admits a vector space partition into one space of dimension $d<t$ and $(q^t-q^d)/(q-1)$ spaces of dimension $1$, such a vector space partition $\mathcal P$ always exists if $a\geq m_d$.

As we will see in Section \ref{sec:5}, it is not at all necessary that $a\geq m_d$. 
So the problem is then to find the maximum value of $m_d$ for a given value of $a$
or 
to find the minimum value $a$ for a given number $m_d$ of spaces of dimension $d$.

The motivation for this particular interest in the relation between the parameters $a$ and $m_d$ comes from the study of maximal partial $t$-spreads in $V(2t,q)$.

\subsection{Relations to maximal partial $t$-spreads}\label{sec:4.1}

When $a<m_d$, we can sometimes assure that the set of spaces of dimension $t$ cannot be extended to a $t$-spread. This will be discussed now.

We will need the following triviality.

\begin{pro}\label{pro:1} Let $W$ be any subspace of $V(2t,q)$. If $\{U_{i}\}_{i=1}^{q^t+1}$
is a $t$-spread of $V$, then 
thoses spaces $U_{i}\cap W$ 
that are not the zero space 
will together constitute a
partition of $W$. \end{pro}

We will also need some lower bounds on the number
of spaces in a vector space partition. The bound in the following proposition improves a bound
of Spera \cite{spera}, who proved that the number of spaces in any non-trivial vector
space
partition is at least equal to $q^d+1$, where $d$ is the least integer with $m_d\neq0$.

\begin{pro}\label{pro:card-part-dep-dimension} Let $d$ be the highest
and $d'<n-d$ one of the other
dimensions that appears in a vector space partition $\mathcal{P}$
of a space $V(n,q)$. Then \begin{equation}
|{\mathcal{P}}|\geq q^{d}+q^{d'}+1\,.\label{eq:7}\end{equation}
\end{pro}
\begin{proof}
By Lemma \ref{lem:1}, if $U$ and $U'$ are spaces of $\mathcal P$ of dimension $d$ and $d'$, respectively, where $d+d'<n$, then there is at least one hyperplane that contains both $U$ and $U'$. Now use Equation \eqref{eq:sec-pack-cond}.
\end{proof}
 If a vector space partition ${\mathcal{P}}$ of $V(n,q)$ consists of spaces solely of dimension
$d$ and
$n-d\leq d$, then the dimension condition implies that there is just one space of dimension $d$ in ${\mathcal P}$, and hence, by using the first packing condition we can deduce that
\begin{equation}
|{\mathcal{P}}|=q^{d}+1\,.\label{eq:8}\end{equation} 
In general, and as easily seen from the packing condition in case every member of ${\mathcal P}$ has dimension at most equal to $n/2$, and in the other cases as an immediate consequence of the two preceding equations,
\begin{equation}
|{\mathcal{P}}|\geq q^{\lceil n/2\rceil}+1\,.\label{eq:part-card-gen}\end{equation}

We will prove in Section \ref{sec:5.2} that the above three bounds for the number of spaces in a vector space partition are tight.

\begin{pro}\label{corol:1}
Let the number of spaces of dimension $t$ in a vector space partition $\mathcal P$ 
of $V(2t,q)$ be $q^{t}+1-a$, and assume that $\mathcal P$ contains $m_d$ spaces of dimension $\frac{t}{2}	<d<t$.
If the $t$-dimensional spaces of $\mathcal P$ 
can be completed to a $t$-spread $\mathcal S$, then 
\[
a\geq\min(m_d,q^{\lceil d/2\rceil}+1)\,.
\] \end{pro}
\begin{proof}
 There are two possible cases. All spaces
of dimension $d$ in the partition may be subspaces of spaces in
the $t$-spread $\mathcal S$. By the dimension condition, as $d>t/2$, each space of $\mathcal S$ can contain at most one of the subspaces of dimension $d$ of $\mathcal{P}$.
So in this case, we must delete at least $m_d$ spaces from $\mathcal S$ to obtain the $t$-dimensional spaces in $\mathcal{P}$. Hence, $a\geq m_d$ holds.

In the other case, at least one space $W$ of
dimension $d$ in $\mathcal{P}$ is not contained in any of the spaces of $\mathcal S$. By Proposition
\ref{pro:1}, the $t$-spread $\mathcal S$ gives a vector space partition of $W$ into subspaces
contained in the spaces $U_{i}$, $i\in I$, belonging to
$\mathcal S$, and where $|I|$ equals the number of spaces in the vector space partition of $W$. These spaces $U_{i}$ are contained in the set
${\mathcal S}\setminus{\mathcal{P}}$. Hence, Proposition~\ref{pro:card-part-dep-dimension} yields $a\geq|I|\geq q^{\lceil\dim(W)/2\rceil}+1$. \end{proof}

The \textit{deficiency} $\delta$ of a partial $t$-spread $\mathcal S$ in $V(kt,q)$ is the 
integer $\sum_{i=0}^{k-1}q^{it}-|{\mathcal S}|$. The problem is to find a general upper bound for the deficiency of a maximal partial $t$-spread in $V(kt,q)$. 

The following bounds for the deficiency of maximal partial $t$-spreads in $V(kt,q)$ are known, see \cite{goeverts}:

\begin{thm}\label{thm:10} For any maximal partial $t$-spread in $V(kt,q)$ with deficiency $\delta>0$:
\begin{enumerate}
\item $\delta\geq\sqrt{q}+1$ when $q$ is a square;
\item $\delta\geq c_{p}q^{2/3}+1$ when $q=p^{h}$, $h$ odd, $h>2$,
$p$ prime, $c_{2}=c_{3}=2^{-1/3}$ and $c_{p}=1$ for $p>3$;
\item $\delta\geq(q+3)/2$ when $q$ is a prime.
\end{enumerate}\end{thm}

If a partial $t$-spread $\mathcal S$ in $V(2t,q)$ cannot be completed to a spread, then $\mathcal S$ will be contained in a maximal partial spread ${\mathcal S}'$ with nonzero deficiency. A consequence of Proposition \ref{corol:1} will thus be: 
\begin{cor}\label{corol:3}
Assume that there is a partition $\mathcal{P}$ of $V=V(2t,q)$ of
which $q^{t}+1-a$ of the subspaces in $\mathcal{P}$ have dimension
$t$.
Let the number of spaces
of dimension $d$, with $t/2 < d < t$, in $\mathcal P$ be denoted by $m_d$, and assume that $m_d\leq q^{\lceil d/2\rceil}+1$.

If $a<m_d$,
then there exists a maximal partial $t$-spread in $V$ with a deficiency
less than or equal to $a$, but larger than $0$. \end{cor}

\subsection{A lower bound on $a$}\label{sec:4.2}

Define 
\[
\mathrm{R}_q(t,d,m):=m(m-1)\frac{\frac{1}{2}(q^{2t-2d}-1)+1-q^{t-d}}{q^{t}-1-m(q^{t-d}-1)}\,.
\]

\begin{thm}\label{thm:lowbound-a} Assume that there is a partition of $V(2t,q)$ of
type $(m_t,\ldots,m_1)$, with $m_t=q^t+1-a$. Let $d<t$ such that $m_d>0$. 

If \begin{equation*}
m_d<\frac{q^{t}-1}{q^{t-d}-1}\,,\end{equation*}
 then \begin{equation*}
a\geq m_d-\mathrm{R}_q(t,d,m_d)\,.\end{equation*}
 \end{thm}

\begin{proof} 
As no two spaces of dimension $t$ in the partition can be contained
in the same hyperplane, $s_{b}\neq0$ implies $b_{t}\in\{0,1\}$.
Hence, from Equation~\eqref{eq:sum-sb-i}, we obtain 
\begin{equation*}
m_{t}\frac{q^{t}-1}{q-1}=\sum_{\substack{b\in B\\
b_{t}=1}}s_{b},
\end{equation*}
 and thus by Equation~\eqref{eq:sum-sb-0}, we have \begin{equation*}
\sum_{\substack{b\in B\\
b_{t}=0}}s_{b}=\frac{q^{2t}-1}{q-1}-m_{t}\frac{q^{t}-1}{q-1}=a\frac{q^{t}-1}{q-1}\end{equation*}
and in particular, \begin{align}
\sum_{\substack{b\in B\\
b_{t}=0\\
b_{d}\in\{1,2\}}}s_{b} & \leq a\frac{q^{t}-1}{q-1}.\label{eq:6}\end{align}

As $i-\binom{i}{2}$ is 1 if $i\in\{1,2\}$ and at most 0 otherwise,\[
\sum_{b\in B}b_{d}s_{b}-\sum_{b\in B}\binom{b_{d}}{2}s_{b}\leq\sum_{\substack{b\in B\\
b_{d}\in\{1,2\}}
}s_{b}\]
 holds. We also have\[
\sum_{\substack{b\in B\\
b_{d}\in\{1,2\}\\
b_{t}=1}
}s_{b}\leq\sum_{\substack{b\in B\\
b_{t}=1}
}b_{d} s_{b}=\sum_{b\in B}b_{d}b_{t}s_{b}.\]

So the above two equations together with Equations~\eqref{eq:sum-sb-i},
\eqref{eq:sum-sb-2i} and \eqref{eq:sum-sb-i-j} give \[
\sum_{\substack{b\in B\\
b_{d}\in\{1,2\}\\
b_{t}=0}
}s_{b}\geq m_{d}\frac{q^{2t-d}-1}{q-1}-\binom{m_{d}}{2}\frac{q^{2t-2d}-1}{q-1}-m_{d}m_{t}\frac{q^{t-d}-1}{q-1}.\]
 Now, combining this equation with Equation~\eqref{eq:6}
and using the fact that $m_{t}=q^{t}+1-a$, we obtain \[
a\frac{q^{t}-1}{q-1}\geq m_{d}\frac{q^{2t-d}-1}{q-1}-\binom{m_{d}}{2}\frac{q^{2t-2d}-1}{q-1}-(q^{t}+1-a)m_{d}\frac{q^{t-d}-1}{q-1},\]
 whence the statement of the theorem follows. 
\end{proof}

In the next example we evaluate the bound on $a$ that we get from Theorem \ref{thm:lowbound-a} in case $d$ is close to $t$.

\begin{example}\label{ex:12} Let $d=t-k$. We claim that if 
\begin{equation}
m_{d}\leq \sqrt{2}q^{(t-2k)/2}\,,\label{eq:sqrt}\end{equation}
then \begin{equation}
a\geq m_{d}\,\label{eq:ageq}.\end{equation}
To prove this for the nontrivial case $t/2 < d$, we note that from Theorem \ref{thm:lowbound-a} we get \begin{equation*}
a\geq m_d-\mathrm{R}_q(t,d,m_d)= m_d-m_d(m_d-1)\frac{\frac{1}{2}(q^{2k}-1)+1-q^{k}}{q^{t}-1-m_d(q^{k}-1)}\,.
\end{equation*}
Trivial calculations show that $\mathrm{R}_q(t,d,m_d)<1$ if and only if
\[
 m_d(m_d-1)\frac{1}{2}(q^{2k}-1)+m_d(m_d-2)(1-q^k)<q^t-1\,.
\]
The cases $m_d=1$ and $m_d=2$ are trivial, so we may assume that $m_d\geq3$.
Then the left side of the inequality above is less than
\[
 m_d(m_d-1)\frac{1}{2}(q^{2k}-1)\,,
\]
which is less than 
\[
m_d^2\frac{1}{2}(q^{2k}-1)\,.
\]
Hence, Equation \eqref{eq:ageq} now follows from the condition in Equation \eqref{eq:sqrt}.
\end{example}

The next example shows the consequences of Theorem \ref{thm:lowbound-a} in one of the smallest possible cases.
\begin{example}\label{ex:lowbound-a} Let $q=2$, $t=4$, and $d=3$. We
then get 
\begin{equation}
\mathrm{R}_2(4,3,m_{3})=m_3(m_3-1)\frac{1/2}{15-m_3}\,.\label{eq:formel}\end{equation}
For $m_3\leq4$ this implies $a\geq m_3$ by Theorem~\ref{thm:lowbound-a}. 

If $m_3=5$, then we can from the above formula just deduce that $a\geq 4$. If there were a vector space partition $\mathcal P$ of type $(13,5,0,25)$ in $V(8,2)$, then, by Corollary \ref{corol:3}, the spaces of dimension 4 in $\mathcal P$ either would constitute a maximal partial 4-spread in $V(8,2)$ with deficiency $\delta=4$ or could be completed to a maximal partial 4-spread with deficiency $\delta=3$. So we cannot use Theorem \ref{thm:10} to exclude this type of partition. However, we can exclude the existence of this type of partition using the necessary conditions in Section \ref{sec:3}, as from \cite{latte} the corresponding polytope contains no integer points.

If $6\leq m_3\leq 8$, then we get from Equation \eqref{eq:formel} and Theorem \ref{thm:lowbound-a} that $a\geq 5$. Thus, no partition of type $(13,6,0,18)$ of $V(8,2)$ exists.
In Section \ref{sec:5} we will show that there actually exist vector space partitions of type $(12,6,0,33)$ in $V(8,2)$
as well as of types $(12,7,0,26)$ and $(12,8,0,19)$, that is, with $a=5$.

For $m_3=9$, the above formula gives $a\geq 4$, which shows that for $m_3$ large, Theorem \ref{thm:lowbound-a} is not stronger than the packing condition, as, most trivially, there is no vector space partition of type $(13,9,0,-3)$ in $V(8,2)$.
\end{example}

In any case, the example above shows that our new necessary conditions in Section \ref{sec:3} complete the necessary conditions mentioned in Section \ref{sec:1} and Section \ref{sec:2}, as the existence of a partition of type $(13,6,0,18)$ in $V(8,2)$ is not possible to exclude neither by the first packing condition nor by the dimension condition nor by the tail condition.

The following corollary is an immediate consequence of Theorem \ref{thm:lowbound-a}: 
\begin{cor}
Assume that $V(2t,q)$ has a partition $\mathcal{P}$
of type $(m_t,\ldots,m_1)$ with $m_t=q^t+1-a$. Let $d$ be any integer with $t/2<d<t$ and let $m$ be the number of spaces of $\mathcal{P}$ of a dimension at least equal to $d$ and less than $t$.
If \begin{equation*}
m<\frac{q^{t}-1}{q^{t-d}-1}\,,\end{equation*}
 then \begin{equation*}
a\geq m-\mathrm{R}_q(t,d,m)\,.\end{equation*}
 \end{cor} 
\begin{proof}
We derive a partition ${\mathcal{P}}'$ from
${\mathcal{P}}$ simply by partitioning each space $W$ of $\mathcal{P}$ of dimension
larger than $d$ and less than $t$ into one space of dimension
$d$ and the other spaces of dimension 1, for example. The partition $\mathcal{P}'$ has $m$ spaces of dimension $d$. 

Now apply Theorem \ref{thm:lowbound-a}.
\end{proof}

\section{Constructions and examples}\label{sec:5}

We will discuss whether or not the bounds given in Proposition \ref{pro:card-part-dep-dimension} and Theorem \ref{thm:lowbound-a} are tight. For that purpose, we will use two explicit constructions of vector space partitions. 

Both constructions originate in the following $t$-spread $\mathcal{P}$ of $V(kt,q)$, given in \cite{bu}: 

Identify $V(kt,q)$ with the finite field $\GF(q^{kt})$. This field has a subfield with $q^t$ elements, which we denote by $\GF(q^t)$. Now, simply, let
\[
\mathcal{S}=\{\,\beta \GF(q^t)\colon\,\beta\in \GF(q^{kt})\,\}\,.
\]
Note that 
\begin{equation*}
\beta \GF(q^t)=\beta'\GF(q^t)\qquad\Longleftrightarrow\qquad\beta'\beta^{-1}\in \GF(q^t)\,.
\end{equation*}

The constructions below will also provide us with examples of partitions $\mathcal P$, as described in the beginning of Section \ref{sec:4}, such that $a<m_d$. 
All vector space partitions $\mathcal P$ of $V(kt,q)$ that we will construct in this section have the property that the 
spaces of dimension $t$ in $\mathcal P$ are contained in the $t$-spread $\mathcal{S}$.

\subsection{Construction I}

Let $A=\{\alpha_1,\alpha_2, \dots, \alpha_l\}$
 be a set of elements of $\GF(q^{kt})$ that are linearly independent over $\GF(q^t)$. 

We will henceforth consider $\GF(q^{kt})$ as a vector space over $\GF(q)$ and we will let $\bar A$ denote the linear span of the elements of $A$, without the zero vector. We call two elements $\gamma,\gamma'$ of $\GF(q^{kt})$ \emph{parallel}, if $\gamma/\gamma'\in \GF(q)$.
\begin{lem}\label{lem:parallel-classes}
For any two non parallel elements $\bar\alpha$ and $\bar\alpha'$ of $\bar A$, 
\[
\bar\alpha \GF(q^t)\cap \bar\alpha'\GF(q^t)=\{0\}\,.
\]
\end{lem}
We note that if $\bar\alpha$ and $\bar\alpha'$ are parallel, then $\bar\alpha \GF(q^t)=\bar\alpha'\GF(q^t)$, since $\GF(q)\subset \GF(q^t)$.

\begin{proof}
If the conclusion of the lemma does not hold, then $\bar\alpha \GF(q^t)=\bar\alpha'\GF(q^t)$, and so for some $\gamma\in \GF(q^t)$ we have
\[
\sum_{i=1}^l s_i\alpha_i=\bar\alpha=\gamma\bar\alpha'=\gamma\sum_{i=1}^l s_i'\alpha_i,
\] 
with $s_i,s_i'\in \GF(q)$.
As the elements of $A$ are linearly independent, we get
\[
s_i=\gamma s_i'\neq 0,
\]
for at least one $i$; hence, $\gamma$ is in $\GF(q)$.
\end{proof}

\begin{lem}\label{lem:bar-U_i-subspaces}
For any 1-dimensional subspace $U$ of $\GF(q^t)$, the set $\bar U$ given by
\[
\bar U=\cup_{\bar \alpha\in \bar A} \bar\alpha U\,
\]
is an $l$-dimensional subspace of $V=\GF(q^{kt})$.
\end{lem}
\begin{proof}
Let $u$ be a generator of $U$. If $v$ and $v'$ belong to $\bar U$, then there are $\bar\alpha,\bar\alpha'\in\bar A$ with
 $v=\bar\alpha u$ and $v'=\bar \alpha'u$. Hence, for $r\in \GF(q)$,
\[
v+rv'=(\bar \alpha+r\bar \alpha')u=\bar \alpha''u\,,
\]
where $\bar\alpha''$ is an element of $\bar A$. So 
$v+rv'$ belongs to $\bar\alpha''U$.

It can be easily seen that $\{\alpha_1 u,\alpha_2 u,\ldots, \alpha_l u\}$ is a basis of $\bar{U}$.
\end{proof}

\begin{lem}
Let the subspace $W$ together with the 1-dimensional subspaces $U_i$, for $i\in I$, constitute a vector space partition of $\GF(q^t)$. Then the spaces $\bar \alpha W$, for $\bar\alpha\in \bar A$, together with the spaces $\bar U_i$, where
\[
\bar U_i = \cup_{\bar \alpha\in \bar A} \bar\alpha U_i,
\]
for $i\in I$, will constitute a vector space partition of the set \[\cup_{\bar\alpha\in \bar A}\bar\alpha \GF(q^t)\,.\]
\end{lem}
\begin{proof}
We use the fact that for any non-zero element $\alpha$ of $\GF(q^{kt})$, the map
\[
x\quad\mapsto\quad \alpha x
\]
is a bijective linear map that maps $\beta \GF(q^t)$ onto the space $\alpha\beta \GF(q^t)$. In particular, this is true for all elements of $\bar A$, so the spaces $\bar\alpha W$ and $\bar U_i = \bigcup_{\bar\alpha\in\bar A}\bar\alpha U_i$ form a partition; the sets $\bar U_i$ are subspaces due to Lemma~\ref{lem:bar-U_i-subspaces}.
\end{proof}

From the lemmas above we immediately get a partition of $V(kt,q)$, by substituting the spaces $\bar\alpha \GF(q^t)$, where $\bar\alpha\in\bar A$, in the $t$-spread
${\mathcal S}=\{\beta \GF(q^t)\colon\,\beta\in \GF(q^{kt})\}$, with the spaces described in the preceding lemma. Due to Lemma~\ref{lem:parallel-classes}, the number of substituted spaces equals the number of 
parallel classes in $\bar A$, which is $\frac{q^l-1}{q-1}$. Thus we obtain 

\begin{pro}\label{pro:5}
If $k\in \mathbb{N}_{\geq 2}, l\in \{1,2,\ldots,k\}$ and $W$ is a subspace of $\GF(q^t)$, then
there is a partition of $V(kt,q)$ into $\frac{q^{kt}-1}{q^{t}-1}-\frac{q^{l}-1}{q-1}$
subspaces of dimension $t$, $\frac{q^{l}-1}{q-1}$ subspaces of dimension
$\dim(W)$ and $\frac{q^{t}-q^{\dim(W)}}{q-1}$ subspaces of dimension
$l$.
\end{pro}

\subsection{The bounds in Proposition \ref{pro:card-part-dep-dimension} are tight}\label{sec:5.2}

From the construction of Bu \cite{bu} follows that $V(m+m', q)$, where $m\geq m'$, has a vector space partition into $q^m+1$ spaces. Hence, the bound in Equation~\eqref{eq:8} is tight. The following example shows that the bound in Equation~\eqref{eq:7} in general cannot be improved:

\begin{example}
By Proposition \ref{pro:5} with $k=l=2, t=4$, and $\dim(W)= 3$, there is a partition $\mathcal P$ of $V(8,q)$ of type $(q^4+1-(q+1),q+1,q^3,0)$. The size of $\mathcal P$ is $q^4+q^3+1$.
\end{example}

\subsection{Construction II}

We consider the spread $\mathcal{S}$ of $\GF(q^{8})$, given by
\[
{\mathcal S}=\{\,\alpha \GF(q^4)\colon\, \alpha\in \GF(q^8)\,\}.
\]

In the course of applying Construction~I to $V(2\cdot2,q^{2})$ with
$W=\{0\}$ and $l=2$, we replace subspaces $L_{1}$, $L_{2}$, ...,
$L_{q^{2}+1}$ in $\mathcal S$ by the subspaces $L_{1}^{\perp}$, $L_{2}^{\perp}$,
..., $L_{q^{2}+1}^{\perp}$. They are all of dimension $4$ over $GF(q)$;
their intersections $L_{i}\cap L_{j}^{\perp}=:Q_{i,j}$ are 2-dimensional
over $GF(q)$ and $\{L_{j}^{\perp}\colon\, j\in\{1,2,\ldots,q^{2}+1\}$
partitions $\bigcup_{i}L_{i}$. 
These can be chosen such that $L_1=\GF(q^4)$ and $Q_{1,1}=\GF(q^2)$, which we will assume from now on. By Construction~I, there are $\alpha_i\in \GF(q^8)$ such that $L_i=\alpha_i L_1$, for $i\in\{2,3,\ldots,q^2+1\}$, and $Q_{i,j}=\alpha_i Q_{1,j}$.

Let $S_1$ be a 3-dimensional subspace of $L_1$ containing the space $Q_{1,1}$; let $\gamma_1=\alpha_1=1$.

\begin{lem}\label{lem:17}
 We can choose $\gamma_j\in \GF(q^4)$, for $j\in\{2,3,\ldots,q^2+1\}$, such that
 \begin{equation*}
  S_1 \cap Q_{1,j} = \gamma_j \GF(q)\,,
 \end{equation*}
 and
 \begin{equation*}
  Q_{i,j}=\gamma_j \alpha_i Q_{1,1}\,,
 \end{equation*}
 for all $i,j\in\{1,2,\ldots,q^2+1\}$.
\end{lem}

\begin{proof}
If $j\in\{2,3,\ldots,q^2+1\}$, then $S_1\cap Q_{1,j}$ must be at most 1-dimensional. Assuming it is 0-dimensional for
at least one $j$ and observing that $\{S_1\cap Q_{1,j}\colon\,j\in\{1,2,\ldots,q^2+1\}\}$ is a partition of $S_1$ yields a contradiction
by counting the number of points in $S_1$.

Consequently, since $S_1$ is contained in $L_1=\GF(q^4)$, there must be $\gamma_j\in \GF(q^4)$ for each $j$ such that $S_1\cap Q_{1,j}=\gamma_j \GF(q)$.
Since $\{\gamma Q_{1,1}\colon\,\gamma\in \GF(q^4)\}$ equals $\{Q_{1,j}\colon\, j\in\{1,2,\ldots,q^2+1\}\}$ by Construction~I and is a 
partition of $L_1$, $Q_{1,j}$ must be of this form. As $\gamma_j \GF(q)$ is a subset both of $\gamma_j Q_{1,1}$ and $Q_{1,j}$, we have that $Q_{1,j}=\gamma_j Q_{1,1}$.

Hence $Q_{i,j}$ equals $\alpha_i Q_{1,j}=\alpha_i \gamma_j Q_{1,1}$.
\end{proof}

Let $S_1^\perp$ be a 3-dimensional subspace of $L_1^\perp$ containing $Q_{1,1}$; let $\alpha_1'=1$.

\begin{lem}
For $i\in\{2,3,\ldots,q^2+1\}$, there are $\alpha_i'\in \GF(q^8)$ with
\begin{equation*}
S_1^\perp\cap Q_{i,1}=\alpha_i' \GF(q)\,,
\end{equation*}
for all $i\in\{2,3,\ldots,q^2+1\}$, and
\begin{equation*}
Q_{i,j} = \alpha_i'\gamma_j Q_{1,1}\,,
\end{equation*}
for all $i,j\in\{1,2,\ldots,q^2+1\}$.
\end{lem}
\begin{proof}
As in the proof of the preceding lemma, there are $\alpha_i'$ with $S_1^\perp\cap Q_{i,1}=\alpha_i' \GF(q)$, for
$i\in\{2,3,\ldots,q^2+1\}$. Since $\alpha_i' \GF(q)$ is a subset of both $\alpha_i Q_{1,1}$ and $\alpha_i' Q_{1,1}$, it follows that
$\alpha_i' Q_{1,1}$ equals $\alpha_i Q_{1,1}$, and hence we have $Q_{i,j}=\gamma_j\alpha_i' Q_{1,1}$.
\end{proof}

Note that 
\[L_i = \bigcup_{j=1}^{q^2+1}Q_{i,j}=\bigcup_{j=1}^{q^2+1}\alpha_i'Q_{1,j}=\alpha_i'L_1\,,\]
and
\[L_j^\perp = \bigcup_{i=1}^{q^2+1}Q_{i,j}=\bigcup_{i=1}^{q^2+1}\gamma_j Q_{i,1}=\gamma_j L_1^\perp\,.\]
Hence, $\gamma_j S_1^\perp$ is a subspace of $L_j^\perp$ for all $j$, denoted by $S_j^\perp$.

Similarly, we obtain that $\alpha_i'S_1$ is a subspace of $L_i$ for all $i$, which we denote by $S_i$.

\begin{lem}
For all $i,j\in\{2,3,\ldots,q^2+1\}$,
\begin{equation*}
S_i\cap L_j^\perp=S_j^\perp \cap L_i = \gamma_j\alpha_i' \GF(q).
\end{equation*}
\end{lem}

\begin{proof}
As $S_j^\perp$ is a subspace of $L_j^\perp$, we have that $S_j^\perp \cap L_i$ equals
$S_j^\perp\cap Q_{i,j}=\gamma_j(S_1^\perp\cap Q_{i,1})$, which for the $i$ and $j$ we consider equals 
$\gamma_j\alpha_i' \GF(q)$, by the preceding lemma. The proof of the other equality is similar.
\end{proof}

As in the proof of Lemma \ref{lem:17}, we get that every 3-dimensional subspace of $L_i$ that contains the 2-dimensional space $Q_{i,1}$ intersects each space $Q_{i,j}$, for $j\in\{2,3,\dots,q^2+1\}$, in a 1-dimensional space. Further, if $S_i'$ is a 3-dimensional subspace of $L_i$, containing $Q_{i,1}$, and distinct from $S_i$, then 
\begin{equation*}
S_i'\cap S_i=Q_{i,1}\,.
\end{equation*}
 Hence, for $i,j\in\{2,3,\dots,q^{2}+1\}$, the intersection $S_{i}'\cap S_{i}\cap L_{j}^{\perp}$
is the trivial subspace, which together with \[
S_{i}'\cap S_{j}^{\perp}\subseteq L_{i}\cap S_{j}^{\perp}=S_{i}\cap L_{j}^{\perp}\]
yields $S_{i}'\cap S_{j}^{\perp}=\{0\}$.

Thus, if $\mathcal{P}$ consists of the spaces $S\setminus\{L_i\colon\,i\in\{1,2,\ldots,q^2+1\}\}$, the spaces $S_i'$, for $i\in\{2,3,\ldots,q^2+1\}$, the spaces $S_j^\perp$ for $j\in\{2,3,\ldots,q^2+1\}$,
and all 1-dimensional subspaces of $\bigcup_i L_i$ not contained in any $S_i'$ or $S_j^\perp$, then $\mathcal{P}$ yields
a vector space partition of $\GF(q^8)$. 

Therefore, we have now proved the following proposition:

\begin{pro}\label{pro:6}
For any prime power $q$, there exists a vector space partition $\mathcal P$ of $V(8,q)$ into $q^4-q^2$ subspaces of dimension 4, $2q^2$ spaces of dimension $3$, and the remaining $q^5 - q^4 + q + 1$ spaces of dimension 1.
\end{pro} 

\subsection{On the bound of Theorem \ref{thm:lowbound-a}}
The next example shows that the bound in Theorem \ref{thm:lowbound-a} in general cannot be improved. The example is a continuation of Example \ref{ex:lowbound-a}.

\begin{example} We consider $V(8,2)$. 
By Proposition \ref{pro:6}, there exists a partition of type $(12,8,0,19)$ in $V(8,2)$.
Every vector space has a trivial vector space partition into subspaces of dimension 1. Hence, there are also vector space partitions of $V(8,2)$ of the types $(12,7,0,26)$ and $(12,6,0,33)$, as was premised in Example \ref{ex:lowbound-a}. 
\end{example}

The following example demonstrates that in general the bound that we get from Theorem \ref{thm:lowbound-a} is not sharp.

\begin{example} \label{ex:5}
We consider $V(8,3)$. From Proposition \ref{pro:6} we get that there exists a vector space partition of $V(8,3)$ with $72$ spaces of dimension $4$, $18$ spaces of dimension $3$, and the remaining spaces of dimension $1$. The question is whether we can find a vector space partition with the same number of spaces of dimension $3$ but with more than $72$ spaces of dimension 4. 

With the above parameters, Theorem \ref{thm:lowbound-a} just gives that
\begin{equation*}
a\geq 18-18\cdot17\,\frac{\frac{1}{2}(3^2-1)+1-3}{3^4-1-18(3-1)}>4\,.
\end{equation*}
However, by using Theorem \ref{thm:lowbound-a} we can exclude the existence of a vector space partition of type $(76,8,0,136)$, as by Theorem \ref{thm:lowbound-a}, there can be at most $75$ subspaces of dimension $4$ if the vector space partition contains $8$ spaces of dimension $3$. 
As a consequence of this, and by making use of the trivial partition of a vector space into $1$-dimensional spaces, we deduce that for no integer $i\geq0$ there exists a vector space partition of type $(76,8+i,0,136-i\cdot13)$.  

By Corollary \ref{corol:3}, the construction of a vector space partition of type $(75,9,0,163)$ requires that there exists a partial $4$-spread in $V(8,3)$ of size $75$ that can be embedded in a maximal partial $4$-spread. 
It was proved by Heden, Faina, Marcogini and Pambianco \cite{pambianco} that the maximal size of any maximal partial 2-spread in $V(4,9)$ is 74. Neither this fact nor Theorem \ref{thm:10} exclude the possibilty of a partial $4$-spread of size $75$ that can be embedded in a maximal partial $4$-spread of $V(8,3)$.

We have not yet been able to construct a vector space partition of type $(74,9,0,203)$ in $V(8,3)$. 
\end{example}

\section{Extending the results via derivation}\label{sec:6}

Clearly, there is some kind of relationship between the existence of certain types of vector space partitions and maximal partial $t$-spreads in $V(2t,q)$, as demonstrated in Example \ref{ex:5}. Although maximal partial 2-spreads in $V(4,q)$ have been extensively studied for more than 40 years (the first result was by Mesner \cite{mesner} in 1967), the spectrum of the sizes of maximal partial spreads is known for just a few values of $q$, in fact just for $q\leq5$, and the maximal possible size of a maximal partial spread is known just for $q\leq9$. Consequently, this our study 
indicates that the complete spectrum of 
types of vector space partitions will be extremely hard to find.

In Example \ref{ex:5} we excluded the possibility of the existence of a partition $\mathcal P$ of a certain type by 
considering a partition ${\mathcal P}'$ that we could derive from $\mathcal P$ by splitting up some of the spaces of $\mathcal P$ into subspaces. Using the necessary conditions presented in Section \ref{sec:3}, we found that ${\mathcal P}'$, and hence $\mathcal{P}$, cannot exist. 

Furthermore, observe that the first derivation procedure described above and used in Example \ref{ex:5} gives us the following variant of Theorem \ref{thm:lowbound-a}:
\begin{thm}\label{thm:4} Assume that $V(2t,q)$
has a vector space partition into $q^{t}+1-a$ spaces of dimension $t$ and $m_{d}$ spaces of dimension $d$, for $d=1,2,\dots,t-1$. Let \[
\mu_{d}=\max\left(\left\{ 0,1,\ldots,m_{d}\right\} \cap\left[0,\frac{q^{t}-1}{q^{t-d}-1}\right)\right).\]
 Then we have \[
a\geq\max_{\substack{d\in\{1,\ldots,t-1\}\\
x\in\{1,\ldots,\mu_{d}\}}
}(x-\mathrm{R}_{q}(t,d,x)).\]
\end{thm}

There are additional possibilities to derive vector space partitions from a given vector space partition $\mathcal P$ of a finite vector space $V$. We can take any subspace $U$ of $V$ and consider the intersection of all spaces in $\mathcal P$ with $U$. We then get a vector space partition $\mathcal{P}_U$ of $U$, by letting
\[
{\mathcal P}_U=\{\, U\cap U_i\colon\, U_i\in \mathcal{P},U\cap U_i \neq \{0\}\}\,;
\] 
compare Proposition \ref{pro:1}. The type of ${\mathcal P}_U$ depends on the choice of $U$.

Considering $\mathcal{P}_H$, where $H$ is a hyperplane, ties in with the approach used in Section~\ref{sec:3}. More precisely, assume that a vector space partition of $V(n,q)$ of type $m$ exists. So the integer hull $I(P_m)$ of the associated polytope $P_m$ (compare Section~\ref{sec:3}) is not empty, and \emph{one} point $x\in I(P_m)$ must reflect the derived types obtained by cutting with hyperplanes; that is, $x_b$ of the hyperplanes of $V(n,q)$ are of type $b$. In particular, if $x_b > 0$, then a hyperplane $H$ of type $b$ exists. However, then $\mathcal{P}_H$ is a partition of $V(n-1,q)$ with type $m(b)$, where
\begin{equation*}
{m(b)}_i = b_i + (m_{i+1}-b_{i+1})
\end{equation*}

Thus, we call a type $m$ \emph{feasible} for $V(n,q)$ if it fulfills all known necessary conditions, in particular the one that $I(P_m)$ is nonempty, and furthermore $I(P_m)$ contains a green point. A point $x\in I(P_m)$ is \emph{green}, if $m(b)$ is feasible for $V(n-1,q)$.

Therefore, being feasible is a necessary condition for a type to be realizable as a partition.

In \cite{esssv}, it was proved that for any $n\leq 7$
and $q=2$,
a type $m$ is realizable iff it fulfills all of the following:
\begin{itemize}
\item the first packing condition
\item the dimension condition
\item a weaker version of the tail conditions
\item $m$ is not one of the following exceptions:
	\begin{enumerate}
	\item $n=6$ and $m=(7,3,5)$
	\item $n=7$ and $m=(1,13,7,0)$
	\item $n=7$ and $m=(1,13,6,3)$
	\item $n=7$ and $m=(1,14,3,5)$
	\item $n=7$ and $m=(17,1,5)$.
	\end{enumerate}
\end{itemize}

Hence,
for the purpose of a computer search, 
we implemented
the recursive check outlined above, 
checking for the first packing condition, the dimension condition and the tail condition at each derivation stage. We found that for the first, fourth and fifth exception, the integer hull of their associated polyhedra is empty; thus there is no need for recursion in these cases. The second exception is excluded after several recursion steps, which demonstrates that feasibility as defined above is stronger than the individual conditions. The third exception survives the complete derivation check.
Thus, we conclude that feasibility
is not sufficient for the realizability of a type.

Finally, the problem of vector space partitions can be generalized as follows:

\begin{question}
Given a ranked finite lattice $L$ with order $\leq$ and least element $0_L$, does there exist a subset $P$ of its elements such that
\begin{itemize}
\item the meet of any two elements $x,y \in P$ is $0_L$, and
\item for each atom $a$ of $L$, there exists $x\in P$ such that $a \leq x$?
\end{itemize}
What if the type of $P$, that is the number of elements in $P$ for each rank, is prescribed?
\end{question}

If $L$ is the subspace lattice of $V(n,q)$, we arrive at our original problem.

\end{document}